\documentclass[11pt]{article}

\usepackage{xspace}
\usepackage{array}
\usepackage{subfigure}
\usepackage{subfig}
\usepackage{xcolor}
\usepackage{epsfig}
\usepackage{gensymb}
\usepackage[english]{babel}
\usepackage{graphicx}
\usepackage{amssymb}
\usepackage{amsmath}
\usepackage{amsthm}
\usepackage{multirow}
\usepackage{todonotes}
\usepackage{enumerate}
\usepackage{paralist}
\usepackage{cases}
\usepackage{dsfont}
\usepackage{siunitx}
\usepackage{fullpage}
\usepackage{lineno}

\newcommand{\remove}[1]{}

\renewcommand{\int}{int}
\newcommand{\N}{\ensuremath{\mathds{N}}} %
\newcommand{\R}{\ensuremath{\mathds{R}}} %

\newcommand{\fullqed}{\hbox{}\nobreak\hfill\ensuremath{\blacktriangleleft}}

\newtheorem{lemma}{Lemma}
\newtheorem{theorem}{Theorem}
\newtheorem{obs}{Observation}
\newtheorem{property}{Property}%

\renewcommand{\phi}{\varphi}

\begin{document}
\title{Universal Geometric Graphs\thanks{Partially supported by the MSCA-RISE project ``CONNECT'' No 734922, the MIUR Project ``AHeAD'' under PRIN 20174LF3T8, and the NSF award DMS-1800734.}}
\author{Fabrizio Frati\thanks{Department of Engineering, Roma Tre University, Rome, Italy. Email: \texttt{frati@dia.uniroma3.it}}
\and Michael Hoffmann\thanks{Department of Computer Science, ETH Z\"urich, Z\"urich, Switzerland. Email: \texttt{hoffmann@inf.ethz.ch}}
\and Csaba D. T{\'o}th\thanks{Department of Mathematics, California State University Northridge, Los Angeles, CA,
and Department of Computer Science, Tufts University, Medford, MA, USA. Email: \texttt{csaba.toth@csun.edu}}
}
\date{}
\maketitle

\begin{abstract}
We introduce and study the problem of constructing geometric graphs that have few vertices and edges and that are universal for planar graphs or for some sub-class of planar graphs; a geometric graph is \emph{universal} for a class $\mathcal H$ of planar graphs if it contains an embedding, i.e., a crossing-free drawing, of every graph in $\mathcal H$.

Our main result is that there exists a geometric graph with $n$ vertices and $O(n \log n)$ edges that is universal for $n$-vertex forests; this extends to the geometric setting a well-known graph-theoretic result by Chung and Graham, which states that there exists an $n$-vertex graph with $O(n \log n)$ edges that contains every $n$-vertex forest as a subgraph. Our $O(n \log n)$ bound on the number of edges cannot be improved, even if more than $n$ vertices are allowed.

We also prove that, for every positive integer $h$, every $n$-vertex convex geometric graph that is universal for $n$-vertex outerplanar graphs has a near-quadratic number of edges, namely $\Omega_h(n^{2-1/h})$; this almost matches the trivial $O(n^2)$ upper bound given by the $n$-vertex complete convex geometric graph.

Finally, we prove that there exists an $n$-vertex convex geometric graph with $n$ vertices and $O(n \log n)$ edges that is universal for $n$-vertex caterpillars.
\end{abstract}

\section{Introduction}\label{sec:intro}

A graph $G$ is \emph{universal} for a class $\mathcal{H}$ of graphs if $G$
contains every graph in $\mathcal{H}$ as a subgraph.  The study of universal
graphs was initiated by Rado~\cite{rad-uguf-64} in the 1960s. Obviously, the
complete graph $K_n$ is universal for any family $\mathcal{H}$ of $n$-vertex
graphs. Research focused on finding the minimum size (i.e., number of edges) of
universal graphs for various families of sparse graphs on $n$ vertices. Babai et al.~\cite{bcegs-gcasg-82} proved that if $\mathcal{H}$ is the family of all
graphs with $m$ edges, then the size of a universal graph for $\mathcal{H}$ is in
$\Omega(m^2/\log m)$ and $O(m^2\log \log m \log m)$. Alon et
al.~\cite{ac-ougde-08,AlonCKRRS00} constructed a universal graph of optimal $\Theta(n^{2-2/k})$ size for $n$-vertex graphs with maximum degree
$k$.

Significantly better bounds exist 
for minor-closed families. Babai et al.~\cite{bcegs-gcasg-82} proved that there exists a universal graph with $O(n^{3/2})$ edges for $n$-vertex planar graphs. For bounded-degree planar graphs, Capalbo~\cite{c-sugbpg-02} constructed universal graphs of linear size, 
improving an earlier bound by Bhatt~et~al.~\cite{bclr-ugbtpg-89}, which extends
to other families with bounded bisection width. B{\"o}ttcher et
al.~\cite{BOTTCHER2010,Bottcher2009} proved that every $n$-vertex graph with
minimum degree $\Omega(n)$ is universal for $n$-vertex bounded-degree planar graphs. For $n$-vertex trees, Chung and Graham~\cite{cg-gcast-78,cg-ugst-83} constructed a
universal graph of size $O(n\log n)$, and showed that this bound is
asymptotically optimal apart from constant factors.

Rado~\cite{rad-uguf-64} has introduced universality also for induced subgraphs.
A graph $G$ is \emph{induced universal} for a class $\mathcal{H}$ of
graphs if $G$ contains every graph in $\mathcal{H}$ as an induced subgraph. 
Alon and Nenadov~\cite{an-oiugb-19} described a graph on $O(n^{\Delta/2})$ vertices
that is induced universal for the class of $n$-vertex graphs of maximum degree $\Delta$.
Recently, Dujmovi\'c et al.~\cite{dej-alpgb-20} showed,
improving earlier results by Bonamy et al.~\cite{BonamyGP20},
that for every $n\in \mathbb{N}$, there exists a graph $U_n$ with $n^{1+o(1)}$ vertices that contains every $n$-vertex planar graph as an induced subgraph.

In this paper, we extend the concept of universality to geometric graphs. A \emph{geometric graph} is a graph together with a straight-line drawing in the plane in which the vertices are distinct points and the edges are straight-line segments not containing any vertex in their interiors. We investigate the problem of constructing, for a given class $\mathcal H$ of planar graphs, a geometric graph with few vertices and edges that is \emph{universal for $\mathcal H$}, that is, it contains an \emph{embedding} 
of every graph into $\mathcal{H}$.
For an (abstract) graph $G_1$ and a geometric graph $G_2$, an \emph{embedding of $G_1$ into $G_2$} is an injective graph homomorphism $\phi:V(G_1)\rightarrow V(G_2)$ such that (i) every edge $uv\in E(G_1)$ is mapped to a line segment  $\phi(u)\phi(v)\in E(G_2)$; and (ii) every pair of edges $u_1v_2, u_2v_2\in E(G_1)$ is mapped to a pair of noncrossing line segments $\phi(u_1)\phi(v_1)$ and $\phi(u_2)\phi(v_2)$ in the plane.

Previous research in the geometric setting was limited to finding the smallest
\emph{complete} geometric graph that is universal for the planar graphs on $n$ vertices.  The intersection pattern of the edges of a geometric graph is determined by the location of its vertices; hence universal complete geometric graphs are commonly referred to as \emph{$n$-universal point sets}.
De~Fraysseix et al.~\cite{FraysseixPP90} proved that the $2n\times n$ section
of the integer lattice is an $n$-universal point set. Over the last 30 years, the upper bound on the size of an $n$-universal point set has been improved from $2n^2$ to $n^2/4+O(n)$~\cite{BannisterCDE14}; the
current best lower bound is $(1.293-o(1))n$~\cite{Scheucher20} (based on stacked
triangulations, i.e., planar graphs of treewidth three; see also~\cite{CardinalHK15,Kurowski04}). It is known that
every set of $n$ points in general position is universal for $n$-vertex
outerplanar graphs~\cite{Bose02,gmpp-ept-90}. An $O(n^{3/2}\log n)$ upper bound is known for $n$-vertex stacked
triangulations~\cite{FulekT15}.

\paragraph{Our Results.}
The results on universal point sets yield an upper bound of $O(n^4)$ for the
size of a geometric graph that is universal for $n$-vertex planar graphs and $O(n^2)$ for
$n$-vertex outerplanar graphs, including trees. We improve the upper bound for $n$-vertex
trees to an optimal $O(n\log n)$, and show that the quadratic upper bound for outerplanar graphs
is essentially tight for convex geometric graphs. More precisely, we prove the following results:

\begin{itemize}\itemsep 0pt
	\item For every $n\in \N$, there exists a geometric graph $G$ with $n$ vertices
	and $O(n\log n)$ edges that is universal for forests with $n$ vertices
   (Theorem~\ref{thm:main} in Section~\ref{sec:geotree}).
    The bound of $O(n\log n)$ edges is asymptotically optimal, apart from constant factors, even in the abstract setting,
    for caterpillars, and if the universal graph is allowed to have more than $n$ vertices~\cite[Theorem~1]{cg-gcast-78}.
    The proof of universality is constructive and yields a polynomial-time algorithm that embeds any forest with $n$
	vertices into $G$. 
	\item For every $h\in \N$ and $n\geq 3h^2$, every $n$-vertex convex geometric graph that is universal for the family
    of $n$-vertex cycles with $h$ disjoint chords has $\Omega_h(n^{2-1/h})$ edges (Theorem~\ref{thm:h-chords} in Section~\ref{sec:cx}); this almost matches the trivial $O(n^2)$ upper bound, which hence cannot be improved by polynomial factors even for $n$-vertex outerplanar graphs of maximum degree three.
    For $n$-vertex cycles with $2$ disjoint chords, there exists an $n$-vertex convex geometric graph with $O(n^{3/2})$ edges (Theorem~\ref{thm:twochords} in Section~\ref{sec:cx}), which matches the lower bound above.
	\item For every $n\in \N$, there exists a convex geometric graph $G$ with $n$ vertices
	and $O(n\log n)$ edges that is universal for $n$-vertex caterpillars
    (Theorem~\ref{thm:caterpilar} in Section~\ref{sec:cx}).
\end{itemize}

\section{Universal Geometric Graphs for Forests}
\label{sec:geotree}

In this section, we prove the following theorem.

\begin{theorem}\label{thm:main}
	For every $n\in \N$, there exists a geometric graph $G$ with $n$ vertices and
	$O(n\log n)$ edges that is universal for forests with $n$ vertices.
\end{theorem}

\subsection{Construction}
\label{ssec:construction}

We adapt a construction due to Chung and Graham~\cite{cg-ugst-83},
originally designed for abstract graphs, to the geometric setting.
For a given $n\in \N$, they construct a graph $G$ with
$n$ vertices and $O(n\log n)$ edges such that $G$ contains every forest on $n$
vertices as a subgraph. Let us sketch this construction. For simplicity assume
that $n=2^h-1$, for some integer $h\ge 2$. Let $B$ be a complete rooted ordered binary tree on $n$ vertices. A \emph{level} is a set of vertices at the same distance from the root. The levels are labeled $1,\ldots,h$, from the one containing the root to the one that contains the leaves of $B$. A preorder traversal of $B$ (which consists first of the root, then recursively of the vertices in its left subtree, and then recursively of the vertices in its right subtree) determines a total order on the vertices, which also induces a total order on the vertices in each level of $B$. On each level, we call two consecutive elements in this order \emph{level-neighbors};
in particular, any two siblings are level-neighbors.
For a vertex $v$ of $B$, we denote by $B(v)$ the subtree of $B$ rooted at $v$.
The graph $G$ contains $B$ and three additional groups of edges defined as follows
(see \figurename~\ref{fig:1} for an illustration).

\begin{enumerate}[(E1)]\itemsep 0pt
	\item Every vertex $v$ is adjacent to all vertices in the
	subtree $B(v)$;
	\item every vertex $v$ that has a left or right level-neighbor $u$ in $B$
	is adjacent to all vertices in the subtree $B(u)$; and
	\item every vertex $v$ whose parent has a left level-neighbor $p$
	is adjacent to all vertices in the subtree $B(p)$.
\end{enumerate}

\begin{figure}[htb]
	\centering
	\noindent\begin{minipage}[b]{.61\linewidth}
		\noindent\includegraphics[scale=0.9]{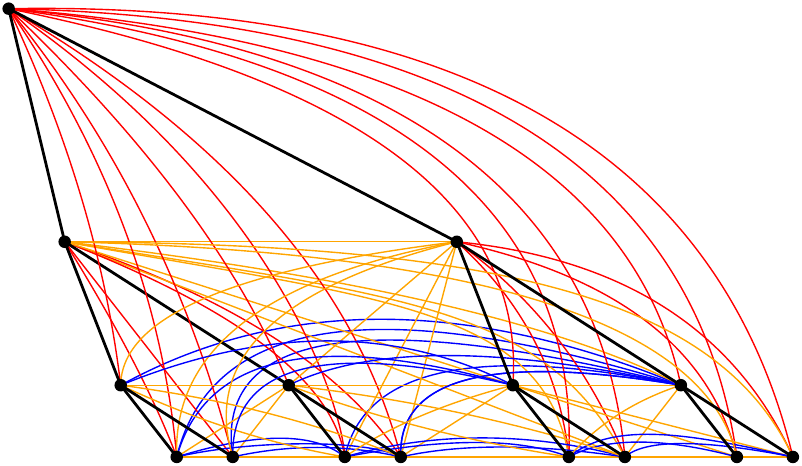}
	\end{minipage}\hfill
	\noindent\begin{minipage}[b]{.38\linewidth}
		\hfill\includegraphics[scale=0.9]{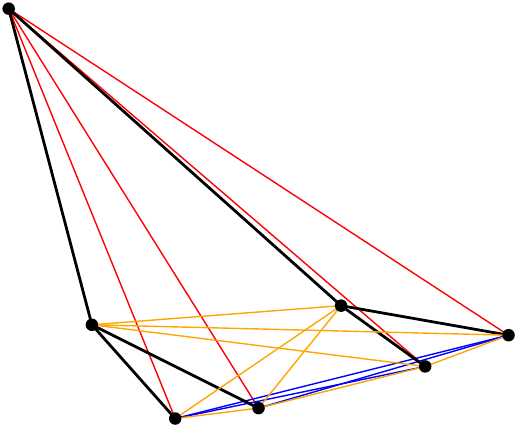}
	\end{minipage}
	\caption{A schematic drawing of the universal graph for $n=15$ vertices (left) and a geometric drawing of the universal graph for $n=7$ vertices. The edges of the tree $B$ are shown black; the edges of the groups (E1), (E2), and (E3) are shown red, orange, and blue, respectively. Edges that belong to several of these groups are shown in the color of the first group they belong to.\label{fig:1}}
\end{figure}

\paragraph{Number of edges.}
It is easily checked that $G$ has $O(n\log n)$
edges.  Indeed, the binary tree $B$ has $2^{i-1}$ vertices on level $i$, for
$i=1,\ldots,h$.  A vertex $v$ on level $i$ has $2^{h-i+1}-1$ descendants (including itself), and its at most two level-neighbors have the same number of descendants. In addition, the left level-neighbor of the parent of $v$ (if
present) has $2\cdot(2^{h-i+1}-1)$ descendants (excluding itself). Altogether
$v$ is adjacent to less than $5\cdot 2^{h-i+1}$ vertices at the
same or at lower levels of $B$. Hence, the number of edges in $G$ is less than
\[
5\cdot \sum_{i=1}^h 2^{i-1} \cdot 2^{h-i+1} = 5\cdot 2^h \cdot h = 5(n+1)\cdot \log_2 (n+1)  \in O(n \log n).
\]

Chung and Graham~\cite{cg-ugst-83} showed that $G$ is universal for forests,
that is, $G$ contains every forest on $n$ vertices as a subgraph.\footnote{In
	fact, the construction by Chung and Graham uses fewer edges: in the edge groups (E2) and (E3) in the definition of $G$, they only use siblings instead of left and right level-neighbors. But we were unable to verify their proof with the smaller edge
	set. Specifically, we do not see why the graph $G_2$ in
	\cite[Fig.~7]{cg-ugst-83} is admissible. However, their proof works with the
	larger edge set we define here.}

\paragraph{Geometric representation.}
We next describe how to embed the vertices of $G$ into $\R^2$;
see \figurename~\ref{fig:assignment}(left) for an illustration.
First, the $x$-coordinates of the vertices are assigned in the order determined by a
preorder traversal of $B$. For simplicity, let us take
these $x$-coordinates to be $0,\ldots,n-1$, so that the
root of $B$ is placed on the $y$-axis. The vertex of $G$ with $x$-coordinate
$i$ is denoted by $v_i$.

The $y$-coordinates of the vertices are determined by a BFS traversal of $B$ starting from the root, in which at every vertex the right sibling is visited before the left sibling. If a vertex $u$ is visited before a vertex $v$ by this traversal, then $u$ gets a larger $y$-coordinate than $v$. The gap between two consecutive $y$-coordinates
is chosen so that every vertex is above every line through two vertices with smaller $y$-coordinates; this implies that, for any vertex $v$, all vertices with larger $y$-coordinate than $v$, if any, see the vertices below $v$ in the same circular order as $v$. The vertices of $G$ are in general position, that is, no three are
collinear.

Our figures display the vertices of $B$ in the correct $x$- and $y$-order,
but---with the exception of \figurename~\ref{fig:1}(right)---they are not to scale.
The $y$-coordinates in our construction are rapidly increasing (similarly to constructions in~\cite{bmn-lbwen-11,FulekT15}).
For this reason, in our figures we draw the edges in $B$ as straight-line segments and all other edges as Jordan arcs.
In particular, we have the following property (see \figurename~\ref{fig:assignment}(right) for an example).
\begin{obs} \label{obs:stretched}
	If $ab, cd\in E(G)$ such that
	(1)~$a$ has larger $y$-coordinate than $b$, $c$, and $d$, and
	(2)~$b$ has smaller or larger $x$-coordinate than both $c$ and $d$, then $ab$ and $cd$ do not cross.
\end{obs}

\begin{proof}
By (2) we can assume, without loss of generality up to a switch of the labels of $c$ and $d$, that $d$ is below the line through $b$ and $c$. By (1) and by construction, $a$ is above the line through $b$ and $c$, hence $ab$ and $cd$ are separated by the line through $b$ and $c$ and thus do not cross.
\end{proof}

\begin{figure}[htb]
	\centering
	\begin{minipage}[b]{.63\linewidth}
		\centering\includegraphics[scale=0.8]{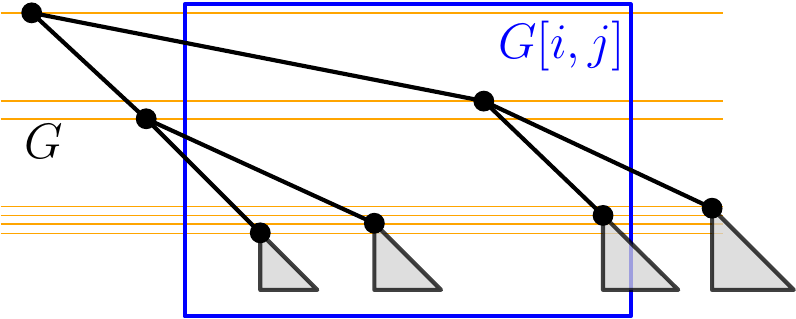}
	\end{minipage}
	\begin{minipage}[b]{.3\linewidth}
		\centering\includegraphics[scale=0.8]{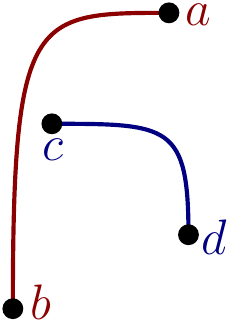}
	\end{minipage}
	\caption{Illustration for the assignment of $x$- and $y$-coordinates to the vertices of $G$, and for the definition of interval (left). Illustration for Observation~\ref{obs:stretched}(right).}
	\label{fig:assignment}
\end{figure}

\subsection{Intervals and Embeddings}
\label{ssec:intervals}

For every interval $[i,j]\subseteq[0,n-1]$ we define $G[i,j]$ to be the subgraph of $G$
that is induced by the vertices with $x$-coordinates in $[i,j]$.
We call the graph $G[i,j]$ an \emph{interval} of $G$.
The \emph{length} of an interval $G[i,j]$ of $G$ is defined as $|G[i,j]|=j-i+1$, which is the number of vertices in $G[i,j]$.
If $I$ is an interval of integers, then we denote by $G(I)$ the corresponding interval of $G$.
For a subset $U\subset V(G)$, we denote by $G[U]$ the subgraph of $G$ induced by $U$.

In Section~\ref{ssec:main}, we present a recursive algorithm that can embed every tree on $h$ vertices into every interval of length $h$ of $G$. In some cases, we embed the root of a tree at some vertex of the interval, and recurse on the subtrees.
For this strategy, it is convenient to embed the root at the center of a spanning star. 
The following lemma shows that every interval contains spanning stars.

\begin{lemma}\label{lem:stars}
	Every interval $G[i,j]$ of $G$ on at least two vertices contains two
	spanning stars: one is centered at the highest vertex $v_k$; another is
	centered at the second highest vertex $v_s$ of $G[i,j]$.
    If $k<j$, then $G[i,j]$	contains a spanning star centered at the highest vertex of
	$G[k+1,j]$ (which may or may not be $v_s$).
\end{lemma}
\begin{proof}
	We first argue for the star centered at $v_k$. By construction (preorder
	traversal and increasing $y$-coordinates along each level from left to right), all vertices in $G[k,j]\subseteq G[i,j]$ belong to $B(v_k)$. By construction, $G$ contains edges from $v_k$ to every vertex in $B(v_k)$.  This completes the proof if $k=i$. Assume that $k>i$. Then $v_k$ has a parent $v_p$; further, we have $p<i$, because $v_k$ is the highest
	vertex of $G[i,j]$ and every vertex is higher than its descendants.
	Therefore, $v_k$ has a left sibling $v_\ell$ (which may or may not be in $G[i,j]$);
	and all the vertices in $G[i,k-1]$ are in $B(v_\ell)$ and hence are adjacent to $v_k$.
    As the vertices are laid out in general position, every star in $G[i,j]$ is noncrossing.
	
	We now argue about the second highest vertex $v_s$ of $G[i,j]$. We consider two cases.
	
	If $k=j$, then $v_s$ is the highest vertex of $G[i,j-1]$. Therefore, $G[i,j-1]$ contains a spanning star
		centered at $v_s$, as argued above if $i<j-1$; if $i=j-1$, then such a star
		trivially exists.  The remaining edge between $v_s$ and $v_k$ exists, as it is
		part of the star centered at $v_k$.
	
	Hence, we may assume that $k<j$. We now show that the highest vertex
		$v_t$ in $G[k+1,j]$ is the center of a spanning star for $G[i,j]$. Note
		that $v_t$ is a child of $v_k$: Namely, since no vertex in $G[i,j]$ is higher than $v_k$, it follows that neither the right level-neighbor of $v_k$, if it exists, nor a vertex on a higher level than $v_k$ are in $G[i,j]$, hence $v_t$ is either the
		left or the right child of $v_k$. Recall that $v_k$ has a parent $v_p$ with
		$p<i$.  Therefore, $v_k$ has a left sibling $v_\ell$, which
		may be in $G[i,j]$ or not; regardless, all the vertices in $G[i,k-1]$ are in $B(v_\ell)$ and are hence adjacent to $v_t$. Further, the edge between $v_t$ and $v_k$ exists, as it is
		part of the star centered at $v_k$. Finally, each vertex in $G[k+1,j]$ is either in $B(v_t)$, hence it is adjacent to $v_t$ as each vertex is adjacent to all its descendants, or in $B(v_q)$, where $v_q$ is the child of $v_k$ different from $v_t$,  hence it is adjacent to $v_t$ as each vertex is adjacent to all the descendants of a sibling.
		
		If $v_s=v_t$, our proof is complete. Let us assume that $v_s$ is in $G[i,k-1]$.
		Then $v_s$ must be on a higher level than $v_{k+1}$, which is the left child of $v_k$, and on a lower level than $v_p$. Hence, $v_s$ is the left sibling of $v_k$, which implies $p=s-1$ and $s=i$. Therefore, $v_s$ is adjacent to all the vertices in $G[i,k-1]$, which are its descendants, as well as to all the vertices of $G[k,j]$, which are descendants of its right-level neighbor $v_k$; hence, $v_s$ is adjacent to all the vertices of $G[i,j]$.
\end{proof}

The recursive algorithm (Section~\ref{ssec:main}) sometimes embeds a subtree of $T$ onto an induced subgraph of $G$ that is ``almost'' an interval, in the sense that it can be obtained from an interval of $G$ by deleting its highest vertex or by replacing its highest vertex with a vertex that does not belong to the interval. Lemma~\ref{lem:deletion} and Lemma~\ref{lem:replacement} below provide the tools to construct such embeddings.

We first prove that the ``structure'' of an interval without its highest vertex is similar to that of an interval; this is formalized by the following definition.
Let $U$ and $W$ be two subsets of $V(G)$ with $h=|U|=|W|$. Let $u_1,\dots,u_h$ and $w_1,\dots,w_h$ be the vertices of $U$ and $W$, respectively, ordered by increasing $x$-coordinates. We say that $G[U]$ and $G[W]$ are \emph{crossing-isomorphic} if the following conditions are satisfied:

\begin{enumerate}[(C1)]\itemsep 0pt
	\item For any two distinct integers $p,q\in\{1,\dots,h\}$, the edge $u_pu_q$ belongs to $G[U]$ if and only if the edge $w_pw_q$ belongs to $G[W]$.
	\item For any four distinct integers $p,q,r,s\in \{1,\dots,h\}$ such that the edges $u_pu_q$ and $u_ru_s$ belong to $G[U]$, the edge $u_pu_q$ crosses the edge $u_ru_s$ if and only if the edge $w_pw_q$ crosses the edge $w_rw_s$.
	\item If $u_i$ is the highest vertex of $G[U]$, for some $i\in \{1,\dots,h\}$, then $w_i$ is the highest vertex of $G[W]$.
\end{enumerate}

In this case, the graph isomorphism given by $\lambda(v_i)=w_i$, for all $i=1,\ldots , n$, is a \emph{crossing-isomorphism}.
Clearly, the inverse of a crossing-isomorphism is a crossing-isomorphism.
We have the following.

\begin{lemma}\label{lem:almost-intervals}
	Let $v_k$ be the highest vertex in an interval $G[i,j]$, and
    assume that $G[i,j]$ contains neither the right
	child of $v_k$ nor any descendant of the left child of its left sibling (if
	it exists). Then $G[i,j]-v_k$ is crossing-isomorphic to some interval $G(I)$ of $G$;
    the interval $I$ can be computed in $O(1)$ time.
\end{lemma}
\begin{proof}
If $k=i$ or $k=j$, then $G[i,j]-v_k$ is the interval $G[i+1,j]$ or $G[i,j-1]$, respectively, and the conclusion is immediate.  Assume that $i<k<j$ and that the subtrees of $B$ rooted at the children of $v_k$ have height $\ell$.  Thus, $B(v_{k+1})$ has $D:=2^\ell-1$	vertices. Let $I=[i-D,j-D-1]$. We prove that $G[i,j]-v_k$ is crossing-isomorphic to $G(I)$. Let $u_1,\dots,u_h$ be the vertices of $G[i,j]-v_k$, ordered by increasing $x$-coordinates; further, let $w_1,\dots,w_h$ be the vertices of $G(I)$, ordered by increasing $x$-coordinates. Refer to \figurename~\ref{fig:crossing-isomorphic}.
	
\begin{figure}[htb]
	\centering\includegraphics[scale=.8]{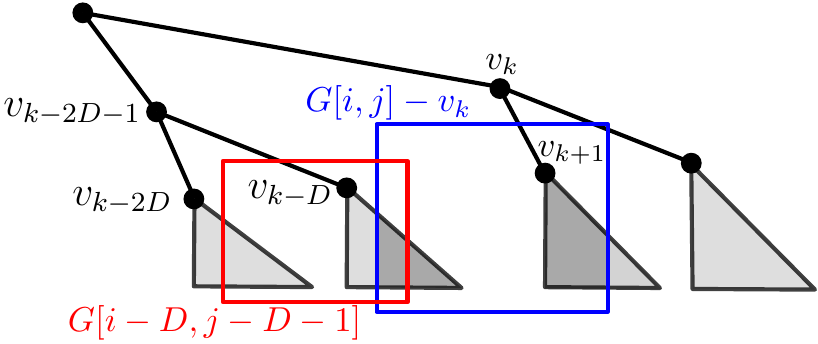}
	\caption{The interval $G[i,j]-v_k$ is crossing-isomorphic to the interval $G[i-D,j-D+1]$.\label{fig:crossing-isomorphic}}
\end{figure}
	
Since $k-1\in [i,j]$ and $v_k$ is the highest vertex in $G[i,j]$, it follows that $v_k$ is not the root of $G$ and $v_{k-1}$ is not the parent of $v_k$, hence $v_k$ is the right child of its parent and has a left sibling $v_{k-2D-1}$. Since $G[i,j]$ does not contain any descendant of the left child of the left sibling of $v_k$, it follows that $G[i,k-1]$ is a subgraph of $G$ induced by vertices in the right subtree of the left sibling of $v_k$. Specifically, $G[i,k-1]$ is induced by the last $k-i<D$ vertices (in a preorder traversal) of the subtree $B(v_{d-D})$ of height $\ell$, rooted at the right child $v_{k-D}$ of $v_{k-2D-1}$. Further, $G[i-D,k-D-1]$ consists of the last $k-i$ vertices (in a preorder traversal) of the  subtree $B(v_{k-2D-1})$ of height $\ell$, rooted at the left child $v_{k-2D}$ of $v_{k-2D-1}$. Hence, $G[i,k-1]$ is crossing-isomorphic to $G[i-D,k-D-1]$.
	
Since $v_{k+1}$ is the left child of $v_k$ and since the right child of $v_k$ is not a vertex of $G[i,j]$, it follows that $G[k+1,j]$ is the subgraph of $G$ induced by the first $j-k$ vertices (in a preorder traversal) of the subtree $B(v_{k+1})$ of height $\ell$. Further, $G[k-D,j-D-1]$ consists of the first $j-k$ vertices (in a preorder traversal) of subtree $B(v_{k-D})$ of height $\ell$. Hence, $G[k+1,j]$ is crossing-isomorphic to $G[k-D,j-D-1]$.

\paragraph{Condition~(C1).}
In order to prove that Condition~(C1) is satisfied (that is, $G[i,j]-v_k$ is isomorphic to $G[I]$), it remains to argue about the edges of $G[i,j]-v_k$ between $G[i,k-1]$ (i.e., vertices in $\{u_1,\dots,u_{k-i}\}$) and $G[k+1,j]$ (i.e., vertices in $\{u_{k-i+1},\dots,u_h\}$), and about the edges of $G(I)$ between $G[i-D,k-D-1]$ (that is, vertices in $\{w_1,\dots,w_{k-i}\}$) and $G[k-D,j-D-1]$ (that is, vertices in $\{w_{k-i+1},\dots,w_h\}$). Consider an edge $u_pu_q$ in $G[i,j]-v_k$, where $p\in[1,k-i]$ and $q\in[k-i+1,h]$. Since a vertex in $G[i,k-1]$ is neither an ancestor nor a descendant of a vertex in $G[k+1,j]$, it follows that $u_pu_q$ belongs to the edge group~(E2) or~(E3). We distinguish between two cases:
\begin{itemize}
\item Suppose that $u_p$ is a descendant of the left level-neighbor $u_l$ of $u_q$, for some $l\in [1,p-1]$. Then $w_l$ is the left level-neighbor of $w_q$ and $w_p$ is a descendant of $w_l$. Hence, the edge $w_pw_q$ belongs to the edge group~(E2). Similarly, if $u_q$ is a descendant of the right level-neighbor $u_r$ of $u_p$, for some $r\in [k-i+1,q-1]$, then $w_r$ is the right level-neighbor of $w_p$ and $w_q$ is a descendant of $w_r$. Hence, the edge $w_pw_q$ belongs to the edge group~(E2).
\item Suppose that $u_p$ is a descendant of the left level-neighbor $u_l$ of the parent $u_r$ of $u_q$, for some $l\in [1,p-1]$ and $r\in [k-i+1,q-1]$. Then $w_r$ is the parent of $w_q$; further, $w_l$ is the left level-neighbor of $w_r$; finally, $w_p$ is a descendant of $w_l$. Hence, the edge $w_pw_q$ belongs to the edge group~(E3).
\end{itemize}

Analogously, given an edge $w_pw_q$ of $G(I)$ with $p\in[1,k-i]$ and $q\in[k-i+1,h]$, the edge $u_pu_q$ belongs to $G[i,j]-v_k$. This concludes the proof that Condition~(C1) is satisfied. 	

\paragraph{Condition~(C2).}
We now prove that Condition~(C2) is satisfied. The proof exploits the following property: a vertex $u_p$ of $G[i,j]-v_k$ is visited before a vertex $u_q$ of $G[i,j]-v_k$ in the BFS traversal of $B$ that determines the $y$-coordinate of the vertices of $G$ if and only if the vertex $w_p$ of $G(I)$ is visited before the vertex $w_q$ of $G(I)$ in the same traversal. This property follows from the fact that the vertices $u_1,\dots,u_{i-k}$ (the vertices  $w_1,\dots,w_{i-k}$) are the last $i-k$ vertices in a preorder traversal of a subtree $B(v_{k-D})$ (resp., of a subtree $B(v_{k-2D})$) of $B$ of height $\ell$, that the vertices  $u_{i-k+1},\dots,u_h$ (the vertices  $w_{i-k+1},\dots,w_h$) are the first $h-i+k$ vertices in a preorder traversal of a subtree $B(v_{k+1})$ (resp., of a subtree $B(v_{k-D})$) of $B$ of height $\ell$, and that the vertices of $B(v_{k-D})$ (resp., of $B(v_{k-2D})$) precede the vertices of $B(v_{k+1})$ (resp., of $B(v_{k-D})$) in a preorder traversal of~$B$.

Consider two arbitrary edges $u_pu_q$ and $u_ru_s$ of $G[i,j]-v_k$. We prove that $u_pu_q$ and $u_ru_s$ cross each other if and only if the edges $w_pw_q$ and $w_rw_s$ of $G(I)$ cross each other. Since the vertices of $G$ are in general position, we can assume that the indices $p$, $q$, $r$, and $s$ are all distinct (as if two among such indices coincide, then $u_pu_q$ and $u_ru_s$ do not cross each other, and neither do the edges $w_pw_q$ and $w_rw_s$). Assume w.l.o.g.\ that $p<q$ and that $p<r<s$. If $q<r$, then $u_pu_q$ and $u_ru_s$ use disjoint $x$-intervals and hence do not cross each other; further, $w_pw_q$ and $w_rw_s$ use disjoint $x$-intervals and hence do not cross each other.

There are two remaining cases to consider, namely $p<r<s<q$ and $p<r<q<s$.

Assume first that $p<r<s<q$. If $u_p$ or $u_q$ is visited first among $u_p$, $u_q$, $u_r$, and $u_s$ in the BFS traversal of $B$ that determines the $y$-coordinates of the vertices of $G$, then $w_p$ or $w_q$ is visited first among $w_p$, $w_q$, $w_r$, and $w_s$ in the same traversal. By construction, $u_p$ or $u_q$ is the highest vertex among $u_p$, $u_q$, $u_r$, and $u_s$; and $w_p$ or $w_q$ is the highest vertex among $w_p$, $w_q$, $w_r$, and $w_s$. By Observation~\ref{obs:stretched}, $u_pu_q$ and $u_ru_s$ do not cross each other and $w_pw_q$ and $w_rw_s$ do not cross each other. Assume hence that $u_r$ or $u_s$, say $u_r$, is visited first among $u_p$, $u_q$, $u_r$, and $u_s$ in the BFS traversal of $B$ that determines the $y$-coordinate of the vertices of $G$; thus, $w_r$ is visited first among $w_p$, $w_q$, $w_r$, and $w_s$ in the same traversal. By construction, $u_r$ is the highest vertex among $u_p$, $u_q$, $u_r$, and $u_s$, and $w_r$ is the highest vertex among $w_p$, $w_q$, $w_r$, and $w_s$. We further distinguish between two cases.

\begin{itemize}
	\item If $u_s$ is visited before $u_p$ and $u_q$ in the BFS traversal of $B$ that determines the $y$-coordinate of the vertices of $G$, then $w_s$ is visited before $w_p$ and $w_q$ in the same traversal. Hence, the $y$-coordinate of $u_s$ is larger than those of $u_p$ and $u_q$ (and the $y$-coordinate of $w_s$ is larger than those of $w_p$ and $w_q$). It follows that $u_pu_q$ and $u_ru_s$ use disjoint $y$-intervals and hence do not cross each other, and similarly $w_pw_q$ and $w_rw_s$ use disjoint $y$-intervals and hence do not cross each other.
	\item If $u_s$ is visited after $u_p$ or $u_q$ (possibly both) in the BFS traversal of $B$ that determines the $y$-coordinate of the vertices of $G$, then $w_s$ is visited after $w_p$ or $w_q$ (possibly both) in the same traversal. By construction, either $u_p$ or $u_q$, whichever is visited first by the BFS traversal of $B$, is assigned a $y$-coordinate large enough so that it lies above the line through $u_s$ and the other point in $\{u_p,u_q\}$ that is visited second by the BFS traversal of $B$. It follows that $u_s$ lies below the line through $u_p$ and $u_q$ and hence the edges $u_pu_q$ and $u_ru_s$ cross each other. Similarly, $w_s$ lies below the line through $w_p$ and $w_q$ and hence the edges $w_pw_q$ and $w_rw_s$ cross each other.
\end{itemize}

Assume next that $p<r<q<s$. Further, assume that the vertex that is visited first among $u_p$, $u_q$, $u_r$, and $u_s$ in the BFS traversal of $B$ that determines the $y$-coordinate of the vertices of $G$ is either $u_p$ or $u_q$, as the case in which it is one of $u_r$ and $u_s$ is analogous (by replacing $p$ with $s$ and $q$ with $r$, respectively). We now distinguish between two cases.

\begin{itemize}
\item Assume that $u_p$ is visited first among $u_p$, $u_q$, $u_r$, and $u_s$ in the BFS traversal of $B$ that determines the $y$-coordinate of the vertices of $G$. Then $w_p$ is visited first among $w_p$, $w_q$, $w_r$, and $w_s$ in the BFS traversal of $B$ that determines the $y$-coordinate of the vertices of $G$. By construction, $u_p$ is the highest vertex among $u_p$, $u_q$, $u_r$, and $u_s$; further, $w_p$ is the highest vertex among $w_p$, $w_q$, $w_r$, and $w_s$.
	\begin{itemize}
	\item If $u_q$ is visited before $u_r$ and $u_s$ in the BFS traversal of $B$ that determines the $y$-coordinate of the vertices of $G$, then $w_q$ is visited before $w_r$ and $w_s$ in the same traversal. Hence, the $y$-coordinate of $u_q$ is larger than those of $u_r$ and $u_s$ and the $y$-coordinate of $w_q$ is larger than those of $w_r$ and $w_s$. It follows that $u_pu_q$ and $u_ru_s$ use disjoint $y$-intervals and hence do not cross each other, and similarly $w_pw_q$ and $w_rw_s$ use disjoint $y$-intervals and hence do not cross each other.
	\item If $u_q$ is visited after $u_r$ or $u_s$ (possibly both) in the BFS traversal of $B$ that determines the $y$-coordinate of the vertices of $G$, then $w_q$ is visited after $w_r$ or $w_s$ (possibly both) in the same traversal. By construction, $u_r$ or $u_s$, whichever is visited first by the BFS traversal of $B$, is assigned a $y$-coordinate large enough so that it lies above the line through $u_q$ and the vertex in $\{u_r,u_s\}$ that is visited second by the BFS traversal of $B$. It follows that $u_q$ lies below the line through $u_r$ and $u_s$ and hence the edges $u_pu_q$ and $u_ru_s$ cross each other. Similarly, $w_s$ lies below the line through $w_p$ and $w_q$ and hence the edges $w_pw_q$ and $w_rw_s$ cross each other.
	\end{itemize}	
\item Assume next that $u_q$ is visited first among $u_p$, $u_q$, $u_r$, and $u_s$ in the BFS traversal of $B$ that determines the $y$-coordinate of the vertices of $G$. Then $w_q$ is visited first among $w_p$, $w_q$, $w_r$, and $w_s$ in the BFS traversal of $B$ that determines the $y$-coordinate of the vertices of $G$. By construction, $u_q$ is the highest vertex among $u_p$, $u_q$, $u_r$, and $u_s$; further, $w_q$ is the highest vertex among $w_p$, $w_q$, $w_r$, and $w_s$. By construction, $u_q$ lies above the straight line through $u_p$ and $u_r$, and above the straight line through $u_p$ and $u_s$. Conversely, the segment $u_ru_s$ lies between these two straight lines. It follows that $u_pu_q$ and $u_ru_s$ do not cross each other. Similarly, $w_q$ lies above the straight line through $w_p$ and $w_r$, and above the straight line through $w_p$ and $w_s$. Conversely, the segment $w_rw_s$ lies between these two straight lines. It follows that $w_pw_q$ and $w_rw_s$ do not cross each other.
\end{itemize}

This concludes the proof that Condition~(C2) is satisfied. 	

\paragraph{Condition~(C3).}
Finally, note that $v_{k+1}$ is the highest vertex of $G[i,j]-v_k$ and $v_{k-D}$ is the highest vertex of $G[i-D,j-D-1]$. Since $v_{k+1}=u_{k-i+1}$ and $v_{k-D}=w_{k-i+1}$, Condition~(C3) follows.

This concludes the proof that $G[i,j]-v_k$ and $G(I)$ are crossing-isomorphic.
\end{proof}

We are now ready to present our tools for embedding trees onto ``almost'' intervals. The first one deals with subgraphs of $G$ obtained by deleting the highest vertex from an interval.

\begin{lemma}\label{lem:deletion}
Let $v_k$ be the highest vertex in an interval $G[i,j]$ with $h+1$ vertices. Suppose that there is a crossing-isomorphism $\lambda$ from $G[i,j]-v_k$ to some interval $G(I)$ of $G$ with $h$ vertices.
Further, suppose that a tree $T$ with $h$ vertices admits an embedding $\phi$ onto $G(I)$. Then $\phi'=\lambda^{-1} \circ \phi$ is an embedding of $T$ onto $G[i,j]-v_k$, and if $a$ is the vertex of $T$ such that $\phi(a)$ is the highest vertex of $G(I)$, then $\phi'(a)$ is the highest vertex of $G[i,j]-v_k$.
\end{lemma}

\begin{proof}
Let $u_1,\dots,u_h$ and $w_1,\dots,w_h$ be the vertices of $G(I)$ and $G[i,j]-v_k$, respectively, ordered by increasing $x$-coordinates. Let $a_1,\dots,a_h$ be the vertices of $T$ ordered so that $\phi(a_i)=u_i$, for $i=1,\dots,h$. Note that $\phi'=\lambda^{-1} \circ \phi$ yields $\phi'(a_i)=w_i$ for $i=1,\dots,h$. We now prove that $\phi'$ is an embedding of $T$ onto $G[i,j]-v_k$ with the properties stated in the lemma.

First, for every $p,q\in \{1,\dots,h\}$ such that $a_pa_q$ is an edge in $T$, we have that $\phi'(a_p)\phi'(a_q)=w_pw_q$ is an edge in $G[i,j]-v_k$. In particular, $\phi(a_p)\phi(a_q)=u_pu_q$ is an edge in $G(I)$, given that $\phi$ is an embedding of $T$ onto $G(I)$. By Condition~(C1) for $\lambda$, $w_pw_q$ is an edge of $G[i,j]-v_k$.

Second, for every $p,q,r,s\in \{1,\dots,h\}$ such that $a_pa_q$ and $a_ra_s$ are distinct edges of $T$, we have that $\phi'(a_p)\phi'(a_q)=w_pw_q$ and $\phi'(a_r)\phi'(a_s)=w_rw_s$ do not cross each other. In particular, $\phi(a_p)\phi(a_q)=u_pu_q$ and $\phi(a_r)\phi(a_s)=u_ru_s$ do not cross each other, given that $\phi$ is an embedding of $T$ onto $G(I)$. By Condition~(C2) for $\lambda$, this implies that $w_pw_q$ and $w_rw_s$ do not cross each other.

Finally, let $a_t$ be the vertex of $T$ such that $\phi(a_t)=u_t$ is the highest vertex of $G(I)$. Condition~(C3) for $\lambda$ implies that $w_t$ is the highest vertex of $G[i,j]-v_k$. By construction, $\phi'(a_t)=w_t$.
This concludes the proof of the lemma.
\end{proof}

The second tool deals with subgraphs of $G$ obtained by replacing the highest vertex of an interval with another ``high'' vertex outside the interval; see \figurename~\ref{fig:replacement} for an illustration.

\begin{lemma}\label{lem:replacement}
Let $G[i,j]$ be an interval of $G$ with $h$ vertices and let $v_k$ be its highest vertex.
Let $v_x$ be a vertex of $G$ that is higher than all vertices in $G[i,j]-v_k$ and that does not belong to $G[i,j]$.
Suppose that a tree $T$ with $h$ vertices admits an embedding $\phi$ onto $G[i,j]$.
Let $a$ be the vertex of $T$ such that $\phi(a)=v_k$; further, let $\phi'(a)=v_x$ and $\phi'(b)=\phi(b)$ for every vertex $b$ of $T$ other than $a$. Then $\phi'$ is an embedding of $T$ onto $G[i,j]-v_k+v_x$.
\end{lemma}

\begin{figure}[htb]
	\centering\includegraphics[scale=.8]{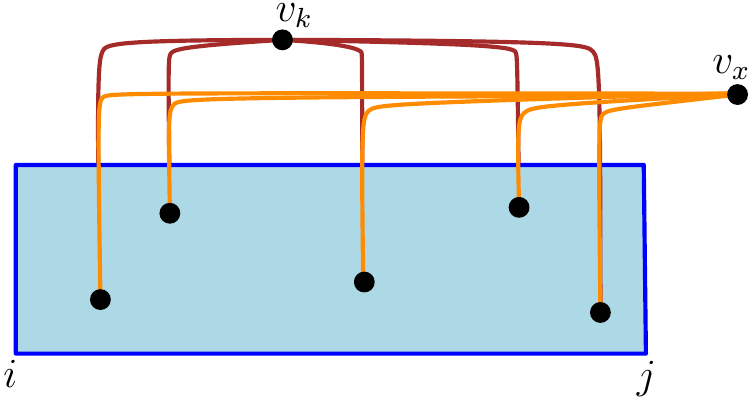}
	\caption{Illustration for the statement of Lemma~\ref{lem:replacement}.}\label{fig:replacement}
\end{figure}

\begin{proof}
By construction, we have $\phi'(b)=\phi(b)$ for every vertex $b$ of $T$ different from $a$; since $\phi$ is an embedding of $T$ onto $G[i,j]$, we only need to prove that each edge incident to $a$ does not cross any other edge of $T$ in $\phi'(T)$. Let $ab$ and $cd$ two edges of $T$, and let $v_p:=\phi'(b)$, $v_q:=\phi'(c)$, and $v_r:=\phi'(d)$, where $p,q,r\in [i,j]$. We prove that $v_xv_p$ and $v_qv_r$ do not cross each other. Assume, w.l.o.g., that $q<r$. Further, assume that $x>j$, as the case in which $x<i$ is symmetric.

\begin{itemize}
	\item If $r<p$, then $v_xv_p$ and $v_qv_r$ do not cross each other as they use disjoint $x$-intervals.
	\item If $p<q$, then $v_xv_p$ and $v_qv_r$ do not cross each other by Observation~\ref{obs:stretched}.
	\item Finally, assume that $q<p<r$. If $v_p$ is visited before $v_q$ and $v_r$ in the BFS traversal of $B$ that determines the $y$-coordinates of the vertices of $G$, then the $y$-coordinate of $v_p$ is larger than the one of $v_q$ and $v_r$, hence $v_xv_p$ and $v_qv_r$ do not cross each other as they use disjoint $y$-intervals. We now show that the case in which $v_p$ is visited after $v_q$ or after $v_r$ (possibly after both) in the BFS traversal of $B$ that determines the $y$-coordinate of the vertices of $G$ does not occur. Suppose the contrary, for a contradiction. It follows that the one between $v_q$ and $v_r$ that is visited first by the BFS traversal of $B$ is assigned a $y$-coordinate large enough so that it lies above the line through $v_p$ and the one between $v_q$ and $v_r$ that is visited second by the BFS traversal of $B$. Hence, $v_p$ lies below the line through $v_q$ and $v_r$. However, this implies $v_kv_p$ and $v_qv_r$ cross each other, contradicting the assumption that $\phi$ is an embedding of $T$ onto $G[i,j]$.
	\end{itemize}

This concludes the proof of the lemma.
\end{proof}

The following lemma is a variant of the (unique) lemma in~\cite{cg-ugst-83}.
It finds a subtree of a certain size in a rooted tree.

\begin{lemma}\label{lem:cut}
	Given a rooted tree $T$ on $m\ge 2$ vertices and an integer $s$, with
	$1\le s\le m$, there is a vertex $c$ of $T$ such that $|V(T(c))|\geq s$ but
	$|V(T(d))|\leq s-1$, for all children $d$ of $c$.
    Such a vertex $c$ can be computed in time $O(m)$.
\end{lemma}
\begin{proof}
	We find vertex $c$ by the following walk on $T$ starting from the root.
	Initially, let $c$ be the root of $T$. While $c$ has a child $d$ such that
	$|V(T(d))|\geq s$, then set $c:=d$. At the end of the while loop,
	$|V(T(c))|\geq s$ but $|V(T(d))|\leq s-1$, for all children $d$ of $c$.
    After precomputing the weight of the subtree $T(v)$ for every vertex $v$ of $T$, 
    the while loop runs in $O(m)$ time.
\end{proof}

\subsection{Proof of Theorem~\ref{thm:main}}
\label{ssec:main}

Given a tree $T$ on $h$ vertices and an interval $G[i,j]$ of length $h$, we describe an algorithm that recursively constructs an embedding $\phi$ of $T$ onto $G[i,j]$. For a subtree $T'$ of $T$, we denote by $\phi(T')$ the image of $\phi$ restricted to the vertices and edges of $T'$. A step of the algorithm explicitly embeds some vertices; the remaining vertices form subtrees that are recursively embedded into pairwise disjoint subintervals of $G[i,j]$. In order to control the interaction between the recursively embedded subtrees and the edges connecting vertices of such subtrees to explicitly embedded vertices, we insist that in every subtree at most two vertices, called
\emph{portals}, are adjacent to \emph{external} vertices (i.e., that are not
part of the subtree). We also ensure that whenever a subtree is embedded onto a
subinterval, the external vertices that connect to the portals of that subtree are embedded above the whole
subinterval.

For a point $p\in \R^2$, we define two quarter-planes: Let $Q^+(p)=\{q\in \R^2: x(p) < x(q) \text{ and } y(p) < y(q)\}$ denote the set of points above and to the right of $p$; similarly, let $Q^-(p)=\{q\in \R^2: x(q)< x(p) \text{ and } y(p) < y(q)\}$ denote the set of points above and to the left of $p$. In Lemma~\ref{lem:main} below, we require that these regions are empty of vertices and edges of the embedded graph, so that they can be used to add edges incident to $p$, when a portal is embedded onto it.

We inductively prove the following lemma, which immediately implies Theorem~\ref{thm:main}
with $G[i,j]=G[0,n-1]$ and a portal $a$ chosen arbitrarily.

\begin{lemma}\label{lem:main}
	We are given a tree $T$ on $h$ vertices, an interval $G[i,j]$ of length $h$, and
	\begin{enumerate}\itemsep 0pt
		\item either a single portal $a$ in $T$ or
		\item two distinct portals $a$ and $b$ in $T$.
	\end{enumerate}
	Then there exists an embedding $\phi$ of $T$ onto $G[i,j]$ with the following properties:
	\begin{enumerate}
		\item If only one portal is given, then
		\begin{enumerate}[(a)]
			\item $\phi(a)$ is the highest vertex in $G[i,j]$; and
			\item\label{cond:quad:1} if $\deg_T(a)=1$ and $a'$ is the unique neighbor of $a$ in $T$,
			then $Q^-(\phi(a'))$ does not intersect any vertex or edge of the embedding $\phi(T(a'))$.
		\end{enumerate}
		\item If two distinct portals are given, then 
		\begin{enumerate}[(a)]
			\item $\phi(a)$ is to the left of $\phi(b)$;
			\item $Q^-(\phi(a))$ does not intersect any edge or vertex of $\phi(T)$; and
			\item $Q^+(\phi(b))$ does not intersect any edge or vertex of $\phi(T)$.
		\end{enumerate}
	\end{enumerate}
\end{lemma}
\begin{proof}
	We proceed by induction on the number of vertices of $T$. In the base case,
	$T$ has one vertex, which must be the portal $a$, and the map $\varphi(a)=v_i$
	maps $a$ to the highest vertex of $G[i,i]$. For the induction step we assume that
	$h\geq 2$ and that the claim holds for all smaller trees.
	
	{\bf Case 1:} There is only one portal $a$.	Let $v_k$ denote the highest vertex in $G[i,j]$. We need to find an
		embedding of $T$ onto $G[i,j]$ where $\phi(a)=v_k$. Consider $T$ to be
		rooted at $a$.  We distinguish two cases depending on the degree of $a$ in
		$T$.
		
	{\bf Case 1.1:} $\deg_T(a)\geq 2$. Assume that $a$ has $t$ children $a_1,\ldots ,a_t$. Refer to
			\figurename~\ref{fig:4}.  Partition the set of integers
			$[i,j]\setminus \{k\}$ into $t$ contiguous subsets $I_1,\ldots ,I_t$ such
			that $|I_x|=|V(T(a_x))|$, for $x=1,\ldots ,t$.  Without loss of generality
			assume that $I_q$ contains $k-1$ or $k+1$, and so $I_q\cup \{k\}$ is an
			interval of integers.
			
			By induction, there is an embedding $\phi_x$ of $T(a_x)$ into $G(I_x)$ such
			that $a_x$ is mapped to the highest vertex of $G(I_x)$, for all $x\neq q$.  Similarly,
			there is an embedding $\phi_q$ of $T-\bigcup_{x\neq q} T(a_x)$ onto $G(I_q\cup \{k\})$ such that
			$a$ is mapped to the highest vertex of $G(I_q\cup \{k\})$, which is $v_k$.  Note that these
			embeddings are pairwise noncrossing since they use pairwise disjoint
			intervals.  Let $\phi:V(T)\rightarrow V(G[i,j])$ be the combination of these
			embeddings. Clearly, both Properties~1(a) and~1(b) are satisfied by $\phi$.
			
			The only edges of $T$ between distinct subtrees among $T(a_1),\dots,T(a_t)$ are of the form $aa_x$, for
			$x\neq q$. The edges $\phi(a)\phi(a_x)$ are in $G[i,j]$ and are pairwise noncrossing by
			Lemma~\ref{lem:stars}. Moreover, $\phi(a)\phi(a_x)$ does not cross $\phi(T(a_x))$, as $\phi(a_x)$ is the highest vertex of $\phi(T(a_x))$ and $\phi(a)$ is higher than $\phi(a_x)$; further, $\phi(a)\phi(a_x)$ does not cross $\phi(T(a_y))$, where $y\neq x$, by Observation~\ref{obs:stretched}. Therefore, $\phi$ is an embedding of $T$ onto $G[i,j]$, as required.\fullqed
		
		\begin{figure}[htb]
			\centering\includegraphics[scale=0.8]{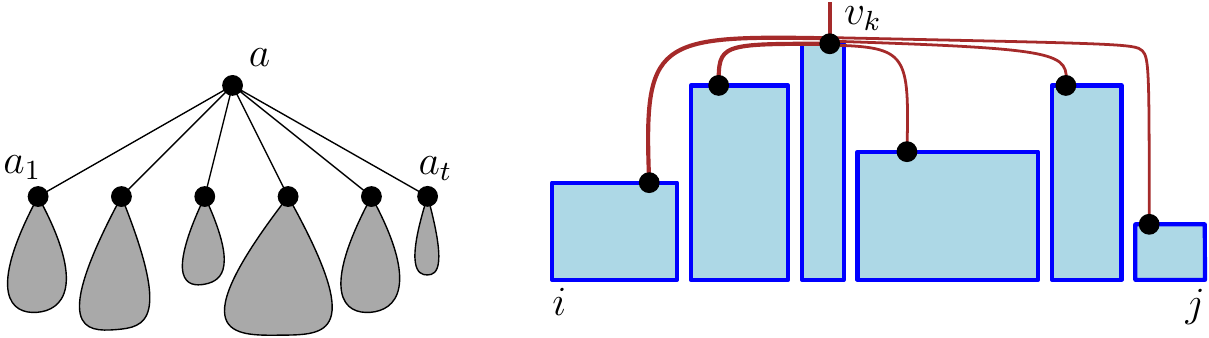}
			\caption{Illustration for Case~1.1: Tree $T$ (left) and its embedding onto $G[i,j]$ (right). \label{fig:4}}
		\end{figure}
		
	{\bf Case 1.2:} $\deg_T(a)=1$. Let $a'$ be the unique neighbor of $a$ in $T$ and let $T'=T(a')$. We need to construct an embedding $\phi$ of $T$ onto $G[i,j]$ such that $\phi(a)=v_k$ and $\phi(T')=G[i,j]-v_k$.
			
	{\bf Case 1.2.1:} $k=j$. Set $\phi(a)=v_k$ and recursively embed $T'$ onto $G[i,k-1]$ with a single portal
				$a'$, which is mapped to the highest vertex in $G[i,k-1]$ (i.e., the second highest vertex in $G[i,j]$). Clearly, both Properties~1(a) and~1(b) are satisfied. Further, $\phi$ is an embedding of $T$ onto $G[i,j]$, since $\phi(a')$ is the highest vertex of $\phi(T')$ and $\phi(a)$ is higher than $\phi(a')$, hence the edge $\phi(a)\phi(a')$ does not cross $\phi(T')$.\fullqed
		
	{\bf Case 1.2.2:} $k=i$. The discussion for this case is symmetric to the one for Case~1.2.1.\fullqed
				
	{\bf Case 1.2.3:} $i<k<j$ and the left sibling $v_\ell$ of $v_k$ exists and is in $G[i,j]$. It follows that $\ell=i$, as if $\ell>i$, then $v_{\ell-1}$, which is the parent of $v_\ell$ and $v_k$, would be a vertex in $G[i,j]$ higher than $v_k$. By construction, $v_i$ is the second highest vertex in $G[i,j]$. Recursively construct an embedding $\psi$ of $T'$ onto $G[i+1,j]$ with a single portal $a'$. By Property~1(a), we have $\psi(a')=v_k$. By Lemma~\ref{lem:replacement}, there exists an embedding $\phi$ of $T'$ onto $G[i+1,j]-v_k+v_i=G[i,j]-v_k$ in which $\phi(a')=v_i$ (hence $\phi$ satisfies Property~1(b)). Finally, set $\phi(a)=v_k$ (hence $\phi$ satisfies Property~1(a)). As in Case~1.2.1, the edge $\phi(a)\phi(a')=v_kv_i$ does not cross $\phi(T')$, hence $\phi$ is an embedding of $T$ onto $G[i,j]$.
	\fullqed
			
	{\bf Case 1.2.4:} $i<k<j$, the left sibling of $v_k$ does not exist or is not in $G[i,j]$, and the right child of $v_k$ is not in $G{[i,j]}$. Refer to \figurename~\ref{fig:no-right-child}. By construction, the left child of $v_k$ is $v_{k+1}$, which is in $G[i,j]$. Since the left sibling of $v_k$ does not exist or is not in $G[i,j]$, and since the right child of $v_k$ is not in $G[i,j]$, it follows that $v_{k+1}$ is the second highest vertex in $G[i,j]$.
				
	\begin{figure}[htb]
		\centering\includegraphics[scale=0.8]{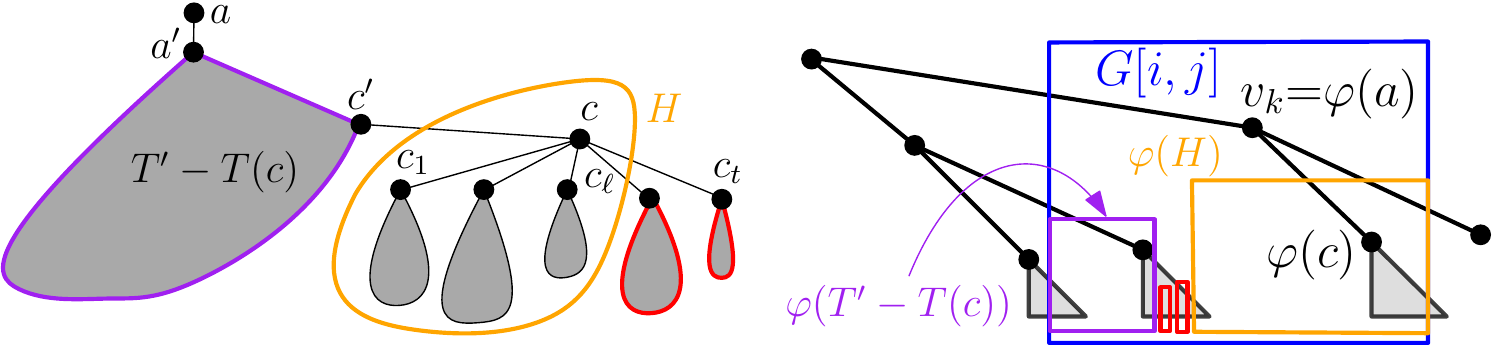}
		\caption{Illustration for Case~1.2.4.  Tree $T$ (left) and its embedding onto $G[i,j]$ (right). \label{fig:no-right-child}}
	\end{figure}

	Set $s=j-k+1$ and note that $s<h$, given that $k>i$. Then Lemma~\ref{lem:cut} yields a vertex $c$ in $T'$ such that $|V(T'(c))|\geq s$ but $|V(T'(d))|\leq s-1$ for all children $d$ of $c$. Label the children of $c$ as $c_1,\ldots, c_t$ in an arbitrary order and
				let $\ell\in [1,t]$ be the smallest index such that
				$1+\sum_{x=1}^\ell |V(T(c_i))|\geq s$.  Since $|V(T(c_\ell))|\leq s-1$, we
				have $s\leq 1+\sum_{x=1}^\ell |V(T(c_i))|\leq 2s-2$.

				Let $c'$ be the parent of $c$ in $T$.  Let $H$ denote the subtree of $T$
				induced by $c$ and the union of $V(T(c_1)),\ldots , V(T(c_\ell))$, and let
				$m=|V(H)|$.  By the above inequalities, we have $s\leq m\leq 2s-2$.  On the one hand,
				$j-k+1\leq m$ implies that the subinterval $G[j-m,j]$ contains $v_k$, and
				so $v_k$ is the highest vertex in $G[j-m,j]$.  On the other hand, the interval $G[j-m,k-1]$ contains $k-1-j+m+1=m-s+1\leq s-1$ vertices, given that $m\leq 2s-2$; however, since the right child of $v_k$ is not in $G{[i,j]}$, we know that the size of a subtree of $B$ rooted at any vertex at the
				level below $v_k$ is larger than or equal to $s-1$. It follows that $G[j-m,j]$ does not
				contain any descendants of the left child of the left sibling of $v_k$ (if
				it exists).  By Lemma~\ref{lem:almost-intervals}, $G[j-m,j]-v_k$ is
				crossing-isomorphic to an interval $G(I)$ of size $m$.
				
				Recursively embed $H$ onto $G(I)$ with one portal $c$, which is mapped to the highest vertex of $G(I)$. By Lemma~\ref{lem:deletion}, there exists an embedding $\phi$ of $H$ onto $G[j-m,j]-v_k$ such that $\phi(c)=v_{k+1}$. We complete $\phi$ into an embedding of $T$ onto $G[i,j]$ as follows. Set $\phi(a)=v_k$ (hence $\phi$ satisfies Property~1(a)). If $c$ has more than $\ell$ children, then embed the subtrees $T(c_{\ell+1}),\ldots , T(c_t)$ on consecutive subintervals to the left of
				$G[j-m,j]$, with single portals $c_{\ell+1},\ldots , c_t$, respectively.  Finally,
				by induction, we can embed $T'-T(c)$ onto the remaining subinterval of
				$G[i,j]$ with two portals $a'$ and $c'$ (hence $\phi$ satisfies Property~1(b), given that the embedding of $T'-T(c)$ satisfies Property~2(b)).
				
				The embeddings $\phi(H),\phi(T(c_{\ell+1})),\dots,\phi(T(c_t)),\phi(T'-T(c))$ are pairwise noncrossing since they use pairwise disjoint intervals. Further, the edges $\phi(c)\phi(c_{\ell+1}),\dots,\phi(c)\phi(c_t)$ belong to $G[i,j]$ by Lemma~\ref{lem:stars} and do not cross each other as they have a common endpoint; further, they do not cross $\phi(H),\phi(T(c_{\ell+1})),\dots,\phi(T(c_t))$ by Observation~\ref{obs:stretched} and do not cross $\phi(T'-T(c))$ since they use intervals disjoint from the one used by $\phi(T'-T(c))$. By analogous arguments, we can conclude that the edge $\phi(a)\phi(a')$ belongs to $G[i,j]$ and does not cross $\phi(H),\phi(T(c_{\ell+1})),\dots,\phi(T(c_t)),\phi(T'-T(c))$, and that the edge $\phi(c)\phi(c')$ belongs to $G[i,j]$ and does not cross $\phi(H),\phi(T(c_{\ell+1})),\dots,\phi(T(c_t))$. Further, the edge $\phi(c)\phi(c')$ does not cross $\phi(T'-T(c))$, since this satisfies Property~2(c) (note that $\phi(c)$ is in $Q^+(\phi(c'))$). It follows that $\phi$ is an embedding of $T$ onto $G[i,j]$.
				\fullqed
			
		{\bf Case 1.2.5:} $i<k<j$, the left sibling of $v_k$ does not exist or is not in $G[i,j]$, and the right child $v_r$ of $v_k$ is in $G{[i,j]}$. By assumption, we have $k+1<r\leq j$; further, the second highest vertex in $G[i,j]$ is $v_r$.
			
		Set $s=j-r+1$. Lemma~\ref{lem:cut} yields a vertex $c$ in $T'$ such that
				$|V(T(c))|\geq s$ but $|V(T(d))|\leq s-1$ for all children $d$ of $c$.
				Let $T(c)$ be the subtree of $T$ rooted at $c$, set $m=|V(T(c))|$,
				and label the children of $c$ by $c_1,\ldots, c_t$ in an arbitrary order.
				Let $c'$ be the parent of $c$ and denote by $T_c(c')$ the subtree of $T$
				induced by $c'$ and $V(T(c))$.
				
				\begin{figure}[htb]
					\centering\includegraphics[scale=0.8]{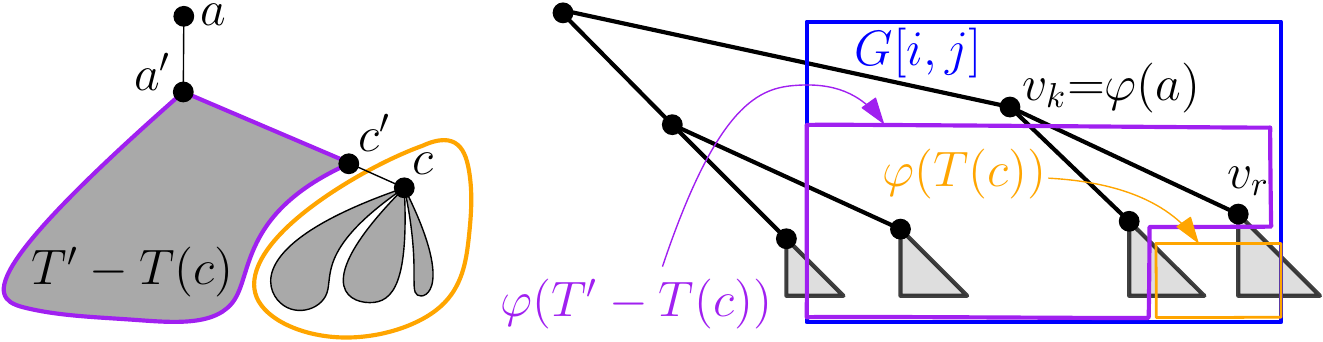}
					\caption{Illustration for Case~1.2.5.1.  Tree $T$ (left) and its embedding onto $G[i,j]$ (right).
						\label{fig:5}}
				\end{figure}
				
{\bf Case 1.2.5.1:} $m\leq j-k-1$. Then the interval $[j-m,j]$ contains $r$ but does not contain $k$, hence $v_r$ is the highest vertex in $G[j-m,m]$. Refer to \figurename~\ref{fig:5}.
                    We construct an embedding $\phi$ of $T$ onto $G[i,j]$ as follows. First, set $\phi(a)=v_k$ (hence $\phi$ satisfies Property~1(a)).
                    	
					By induction, there is an embedding $\psi_1$ of $T'-T(c)$ onto $G[i,j-m-1]$ with two portals $a'$ and $c'$. By Property~2(a) of $\psi_1$, we have that $\psi_1(a')$ is to the left of $\psi_1(c')$; further, by Properties~2(b) and~2(c) of $\psi_1$, we have that neither $Q^-(\psi_1(a'))$ nor $Q^+(\psi_1(c'))$ intersect any vertex or edge of $\psi_1(T'-T(c))$. By Lemma~\ref{lem:replacement}, there is an embedding $\phi$ of $T'-T(c)$ onto $G[i,j-m-1]-v_k+v_r$ (this is part of the embedding $\phi$ of $T$ onto $G[i,j]$). If $\psi_1(a')\neq v_k$, then $\phi$ satisfies Property~1(b), given that $\psi_1$ satisfies Property~2(b). Further, if $\psi_1(a')= v_k$, then $\phi(a')=v_r$ and the only vertex of $G[i,j]$ in the interior of $Q^-(\phi(a'))$ is $\phi(a)=v_k$, hence $\phi$ satisfies Property~1(b).
					
					Again by induction, there is an embedding $\psi_2$ of $T_c(c')$ onto $G[j-m,j]$ with a single portal $c'$. By Property~1(a) of $\psi_2$, we have $\psi_2(c')=v_r$; further, by Property~1(b) of $\psi_2$, we have that $Q^-(\psi_2(c))$ does not intersect any vertex or edge of $\psi_2(T_c(c'))$. Let $\phi(T(c))=\psi_2(T(c))$. Note that $\phi(c')$ may be different from $\psi_2(c')=v_r$. This completes the definition of $\phi(T)$.

					We argue that the edge $\phi(c)\phi(c')$ is present in $G[i,j]$. To simplify the notation, let $v_{p}=\phi(c')$ and $v_q=\phi(c)$, and note that $p<q$ or $p=r$. In the latter case, the edge $\phi(c)\phi(c')$ exists as $\psi_2(c')=v_r$ and the edge $cc'$ belongs to $T_c(c')$; hence, assume that $p<q$. On the one hand, by Property~2(c) of $\psi_1$, we have that $Q^+(\psi_1(c'))$ does not contain any vertex or edge of $\psi_1(T'-T(c))$, hence $k<p$, as otherwise $v_k$ would be in $Q^+(\psi_1(c'))$. It follows that $v_p$ is the highest vertex in $G[p,j-m-1]$. On the other hand, by Property~1(b) of $\psi_2$, we have that $Q^-(\psi_2(c))$ does not contain any vertex or edge of $\psi_2(T(c))$. It follows that $v_q$ is either the highest or the second highest vertex in $G[j-m,q]$ (as $v_r$ might belong to such an interval). Overall, one of $v_p$ or $v_q$ is the highest or the second highest vertex in $G[p,q]$.
					By Lemma~\ref{lem:stars}, $G[p,q]$ contains a star centered at $v_p$ or $v_q$,
					and so it contains the edge $v_p v_q=\phi(c')\phi(c)$, as required.
					
					We now prove that $\phi(T)$ is crossing-free. The embeddings $\phi(T(c))$ and $\phi(T'-T(c))$ are pairwise noncrossing, since they use disjoint intervals, with the exception of the edges incident to $v_r$. However, the edges of $\phi(T'-T(c))$ incident to $v_r$ do not cross $\phi(T(c))$ by Observation~\ref{obs:stretched} and the edges of $\phi(T(c))$ incident to $v_r$ do not cross $\phi(T'-T(c))$ since they use disjoint intervals. Similarly, the edge $\phi(c')\phi(c)$ crosses neither $\phi(T'-T(c))$ nor $\phi(T(c))$ by  Observation~\ref{obs:stretched} and since $Q^+(\phi(c'))$ and $Q^-(\phi(c))$ do not intersect $\phi(T'-T(c))$ and $\phi(T(c))$, respectively. Finally, the edge $\phi(a)\phi(a')$ crosses neither $\phi(T(c))$ nor $\phi(T'-T(c))$, by Observation~\ref{obs:stretched} and since $Q^-(\phi(a'))$ does not intersect any edge or vertex of $\phi(T'-T(c))$, by Property~2(b) of $\phi(T'-T(c))$.\fullqed

	{\bf Case 1.2.5.2:} $j-k-1<m$. In this case, the interval $[j-m,j]$ contains both $k$ and $r$.  Partition the set of
					integers $[j-m,j]\setminus \{k,r\}$ into $t$ contiguous subsets
					$I_1,\ldots ,I_t$ such that $|I_x|=|V(T(c_x))|$, for $x=1,\ldots ,t$.
					Without loss of generality assume that $I_q$ contains $r-1$ or $r+1$.
					
					Let $\mathcal{I}(c)$ be the collection of $t$ sets: $I_q\cup \{v_r\}$ and
					$I_x$, for $x\in [1,t]\setminus \{q\}$. The sets in $\mathcal{I}(c)$ are
					contiguous subsets of $[j-m,j]\setminus \{k\}$. Consequently, at least
					$t-1$ of them are intervals, and at most one of them, say $I_p$, is an interval minus its highest element.  Since every tree
					$T(c_i)$ has at most $s-1$ vertices, we have that $I_p$ has at most
					$s$ elements.  Since $s=j-r-1$, it follows that the size of $I_p$ is smaller than or equal to the size of $B(v_r)$, hence $I_p$ contains neither the right child of $v_k$ nor any descendant of its left sibling. By Lemma~\ref{lem:almost-intervals}, the graph $G(I_p)$ is
					crossing-isomorphic to an interval.
					Therefore, by Lemma~\ref{lem:deletion}, we can embed $T(c_p)$ onto $G(I_p)$. We also recursively embed $T(c_x)$ onto $G[I_x]$ for all
					$x\in [1,t]\setminus \{p,q\}$ and we recursively embed $T(c)-\bigcup_{x\neq q}T_x$ onto
					$G(I_q\cup\{u\})$.
					
					Embed $a$ at $v_k$. By induction, there is an embedding of $T'-T(c)$ onto $G[i,j-m-1]$ with
					portals $a'$ and $c'$. Let $\phi$ be the combination of these embeddings. The proof that $\phi$ is an embedding of $T$ onto $G[i,j]$ satisfying Properties~1(a) and~1(b) is similar to the other cases. In particular, $\phi(c)=v_r$ is the second highest vertex of $G[i,j]$, hence the edge $\phi(c)\phi(c')$ exists by Lemma~\ref{lem:stars}.
					\fullqed
	
	\begin{figure}[htb]
		\centering\includegraphics[scale=0.8]{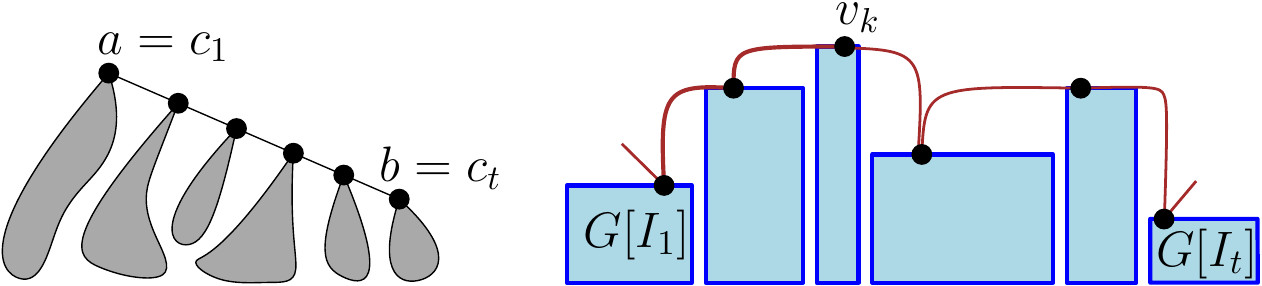}
		\caption{Illustration for Case~2.  Tree $T$ (left) and its embedding onto $G[i,j]$ (right). \label{fig:6}}
	\end{figure}
	
{\bf Case 2:} Two portals. We are given two portals $a$ and $b$; refer to \figurename~\ref{fig:6}.  Let
		$P=(a=c_1,\ldots, c_t=b)$ be the path between $a$ and $b$ in $T$, where
		$t\geq 2$. The deletion of the edges in $P$ splits $T$ into $t$
		trees rooted at $c_1,\ldots ,c_t$. Partition $[i,j]$ into $t$ subintervals
		$I_1,\ldots ,I_t$ such that $|I_x|=|V(T(c_x))|$, for $x=1,\ldots ,t$.
		
		For $x=1,\ldots ,t$,  recursively construct an embedding $\phi_x$ of $T(c_x)$ onto $G(I_x)$ with one portal $c_x$, in which $\phi_x(c_x)$ is the highest vertex in $G(I_x)$. Let $\phi$ be the combination of these embeddings.

		For any distinct $x$ and $y$ in $\{1,\dots,t\}$, we have that $\phi(T(c_x))$ and $\phi(T(c_y))$ do not cross each other, since they use disjoint intervals. Only the edges in $P$ connect vertices from distinct intervals.  For $x=1,\dots,t-1$, let $e_x$ be the edge $\phi(c_x)\phi(c_{x+1})$. Note that $e_x$ is incident to the highest vertex in $G(I_x\cup I_{x+1})$ and hence belongs to $G(I_x\cup I_{x+1})$ by
		Lemma~\ref{lem:stars}. Further, $e_x$ does not cross $\phi(T(c_y))$, with $y<x$ or $y>x+1$, as the intervals spanned by $e_x$ and $\phi(T(c_y))$ are disjoint. Analogously, $e_x$ does not cross any edge $e_y$, with $y\neq x$. Assume that $\phi(c_{x+1})$ is higher than $\phi(c_x)$, the other case is symmetric. Then $e_x$ does not cross $\phi(T(c_x))$, as $\phi(c_x)$ is the highest vertex of $\phi(T(c_x))$ and $\phi(c_{x+1})$ is higher than $\phi(c_x)$; finally, $e_x$ does not cross $\phi(T(c_{x+1}))$, by Observation~\ref{obs:stretched}. Therefore, $\phi$ is an
		embedding of $T$ onto $G[i,j]$.  It is easily checked that this embedding satisfies Properties~2(a), 2(b), and 2(c), as required.
\end{proof}

\section{Convex Geometric Graphs}
\label{sec:cx}

Every graph embedded in a \emph{convex} geometric graph is outerplanar. Clearly, a complete convex geometric graph on $n$ vertices has $O(n^2)$ edges and is universal for $n$-vertex outerplanar graphs. In the next theorem we show that this trivial upper bound is almost tight. For $h\geq 0$ and $n\geq 2h+2$, let $\mathcal{O}_h(n)$ be the family of all outerplanar graphs on $n$ vertices that consist of a spanning cycle plus $h$ pairwise disjoint chords.

\begin{theorem}\label{thm:h-chords}
For every positive integer $h$ and $n\geq 3h^2$, every convex geometric graph on $n$ vertices that is universal for $\mathcal{O}_h(n)$ has $\Omega_h(n^{2-1/h})$ edges.
\end{theorem}

\begin{proof}
The claim trivially holds for $h=1$; we may assume $h\geq 2$ in the remainder of the proof. Let $C$ be a convex geometric graph on $n$ vertices that is universal for $\mathcal{O}_h(n)$, and denote by $\partial C$ its outer (spanning) cycle. The \emph{length} of a chord $uv$ of $\partial C$ is the length of a shortest path between $u$ and $v$ along $\partial C$.
For $k\geq 2$, denote by $E_k$ the set of length-$k$ chords in $C$, and let $m\in \{2,\ldots, \lfloor n/(3h)\rfloor\}$ be an integer such that $|E_m|=\min\{|E_2|,\ldots,|E_{\lfloor n/(3h)\rfloor}|\}$.

Every graph $G\in \mathcal{O}_h$ has a unique spanning cycle $H$, which is embedded onto $\partial C$.
The four endpoints of any two chords of $H$ are noninterleaving along the spanning cycle of $H$,
otherwise $G$ would not be an outerplanar graph. Hence the $h$ chords of $H$
have a well-defined cyclic order along $H$. A \emph{gap} of $G$ is a path between two consecutive chords along $H$ (in cyclic order).

Let $\mathcal{L}$ be the set of \emph{vertex-labeled} outerplanar graphs on $n$ vertices that consist of a spanning cycle $(v_0,\ldots , v_{n-1})$ plus $h$ pairwise-disjoint chords of length $m$ such that one chord is $v_0v_m$ and all $h$ chords have both vertices on the path
$P=(v_0,\ldots , v_{\lfloor n/3\rfloor+hm-1})$. Note that every graph in $\mathcal{L}$ has $h$ gaps. The length of $P$ is $\lfloor n/3\rfloor +hm-1\leq \lfloor n/3\rfloor+h\cdot \lfloor n/(3h)\rfloor-1<2\lfloor n/3\rfloor$. Consequently, the length of the gap between the last and the first chords of $P$ is more than
$n-2\lfloor n/3\rfloor=\lceil n/3\rceil$. This is the longest gap, as the combined lengths of the remaining $h-1$ gaps is at most $\lfloor n/3\rfloor$.

Let $\mathcal{U}$ denote the subset of \emph{unlabeled} graphs in $\mathcal{O}_h(n)$ that correspond to some labeled graph in $\mathcal{L}$. We give lower bounds for $|\mathcal{L}|$ and $|\mathcal{U}|$. Each graph in $\mathcal{L}$ is determined by the $h-1$ gaps between consecutive chords along $P$. The sum of these distances is an integer between $h-1$ and $(\lfloor n/3\rfloor+hm-1)-hm <\lfloor n/3\rfloor$. The number of compositions of $\lfloor n/3\rfloor$ into $h$ positive integers (i.e., $h-1$ distances and a remainder) is $\binom{\lfloor n/3\rfloor}{h-1}\in \Theta_h(n^{h-1})$. Each \emph{unlabeled} graph in $\mathcal{U}$ corresponds to at most two labeled graphs in $\mathcal{L}$, since any graph automorphism setwise fixes the unique spanning cycle as well as the longest gap. Hence, $|\mathcal{U}| \in \Theta(|\mathcal{L}|)\subseteq \Theta_h(n^{h-1})$.

Since $C$ is universal for $\mathcal{O}_h(n)$ and $\mathcal{U}\subset \mathcal{O}_h(n)$,  every graph $G$ in $\mathcal{U}$ embeds onto $C$. Since every embedding of $G$ maps the spanning cycle of $G$ onto the outer cycle $\partial C$ and the $h$ chords of $G$ into a subset of $E_m$, we have that $C$ contains at most $\binom{|E_m|}{h}\leq |E_m|^h$ graphs in $\mathcal{U}$. The combination of the lower and upper bounds for $|\mathcal{U}|$ yields $|E_m|^h \in \Omega_h(n^{h-1})$, hence $|E_m| \in \Omega_h(n^{1-1/h})$. Overall, the number of edges in $C$ is at least $\sum_{i=1}^{\lfloor n/(3h)\rfloor} |E_i|\geq \lfloor n/(3h)\rfloor \cdot |E_m|\in \Omega_h(n^{2-1/h})$.
\end{proof}

For the case $h=2$, the lower bound of Theorem~\ref{thm:h-chords} is the best possible, as shown in the following theorem.

\begin{theorem}\label{thm:twochords}
For every $n\in \N$, there exists a convex geometric graph with $n$ vertices and $O(n^{3/2})$ edges that is universal for $\mathcal{O}_2(n)$.
\end{theorem}
\begin{proof}
We construct a convex geometric graph $C$ and then show that it is universal for $\mathcal{O}_2(n)$.  The vertices
	$v_0,\ldots,v_{n-1}$ of $C$ form a regular convex $n$-gon, and the edges of
	this spanning cycle are in $C$. Denote
	\[
	S=\left\{0,\ldots, \left\lfloor\sqrt{n}\right\rfloor-1\right\}\cup\left\{i
	\left\lfloor\sqrt{n}\right\rfloor\colon 1\le i\le \lfloor\sqrt{n}\rfloor
	\right\}
	\]
	and add (the edges of) a star centered at $v_s$, for every $s\in S$, to
	$C$. Clearly, $C$ contains $O(n^{3/2})$ edges. Moreover, for every
	$d\in\{1,\ldots,\lfloor n/2\rfloor\}$ there exist $a,b\in S$ so that $b-a=d$.
	For any $G\in\mathcal{O}_2(n)$, let $a,b\in S$ so that the distance along the outer cycle between
	the two closest vertices of the two chords of $G$ is
	$b-a$. As $C$ contains stars centered at both $v_a$ and $v_b$, the graph $G$
    embeds onto $C$.
\end{proof}

While we can prove a near-quadratic lower bound for the number of edges of an $n$-vertex convex geometric graph that is universal for $n$-vertex outerplanar graphs, for $n$-vertex trees we only have an $\Omega(n \log n)$ lower bound, which is valid even in the abstract setting and for caterpillars~\cite[Theorem~1]{cg-gcast-78}, where a \emph{caterpillar} is a tree such that the removal of its leaves results in a path, called \emph{spine}.

We next prove that, for caterpillars, the above lower bound is tight. Namely, we construct a convex geometric graph $G$ with $n$ vertices and $O(n \log n)$ edges that is universal for caterpillars with $n$ vertices.
Our construction of $G$ relies on a simple recursive construction of integer sequences.
Specifically, we define a sequence $\pi_n$ of $n$ integers.
Let $\pi_1$ be a one-term sequence $\pi_1=(1)$. For every integer $m$ of the form $m=2^h-1$, where $h\geq 2$,
let $\pi_m = \pi_{(m-1)/2}(m)\pi_{(m-1)/2}$. For example, $\pi_{15}=(1,3,1,7,1,3,1,15,1,3,1,7,1,3,1)$. Now for any $n\in \N$, the sequence $\pi_n$ consists of the first $n$ integers in $\pi_m$, where $m\geq n$ and $m=2^h-1$, for some integer $h\geq 1$. For example, $\pi_{10}=(1,3,1,7,1,3,1,15,1,3)$.
For $i=1,\dots,n$, let $\pi_n(i)$ be the $i$th term of $\pi_n$.

\begin{property}[\cite{fpr-lr-20}] \label{pr:sequence-max}
For every $n\in \N$, $\pi_n$ is a sequence of positive integers such that for every $x$ with $1\leq x \leq n$,
	the maximum of any $x$ consecutive elements in $\pi_n$ is at least $x$.
\end{property}

The graph $G$ has vertices $v_1,\dots,v_n$, placed in counterclockwise order along a circle $c$. Further, for $i=1,\dots,n$, we have that $G$ contains edges connecting $v_i$ to the $\pi_n(i)$ vertices preceding $v_i$ and to the $\pi_n(i)$ vertices following $v_i$ along $c$. We are now ready to prove the following.

\begin{theorem}\label{thm:caterpilar}
For every $n\in \N$, there exists a convex geometric graph $G$ with $n$ vertices and  $O(n\log n)$ edges that is universal for $n$-vertex caterpillars.
\end{theorem}

\begin{proof}
	First, the number of edges of $G$ is at most twice the sum of the integers in $\pi_n$; the latter is less than or equal to the sum of the integers in $\pi_m$, where $m<2n$ and $m=2^h-1$, for some integer $h\geq 1$. Further, $\pi_m$ is easily shown to be equal to $(h-1) \cdot 2^h +1\in O(n \log n)$.
	
	Let $C$ be a caterpillar with $n$ vertices and let $(u_1,u_2,\ldots,u_s)$ be the spine of $C$, for some $s\geq 1$. For $i=1,\ldots,s$, let $S_i$ be the star composed of $u_i$ and its adjacent leaves; let $n_i$ be the number of vertices of $S_i$. Let $m_1=0$; and for $i=2,\ldots,s$, let $m_i=\sum_{j=1}^{i-1} n_j$. For $i=1,\ldots,s$, we embed $S_i$ onto the subgraph $G_i$ of $G$ induced by the vertices $v_{m_i+1},v_{m_i+2},\dots,v_{m_{i}+n_i}$: This is done by embedding $u_i$ at the vertex $v_{x_i}$ of $G_i$ whose degree (in $G$) is maximum, and by embedding the leaves of $S_i$ at the remaining vertices of~$G_i$.
	
	By Property~\ref{pr:sequence-max}, we have that $v_{x_i}$ is adjacent in $G$ to the $n_i$ vertices preceding it and the $n_i$ vertices following it along $c$ (and possibly to more vertices). Hence, $v_{x_i}$ is adjacent to all other vertices of $G_i$, which proves that the above embedding of $S_i$ onto $G_i$ is valid. We now prove that the edge $v_{x_i}v_{x_{i+1}}$ belongs to $G$ for all $i=1,\ldots,s-1$. Again by Property~\ref{pr:sequence-max}, the vertex between $v_{x_i}$ and $v_{x_{i+1}}$ with the highest degree is adjacent to the $n_i+n_{i+1}$ vertices preceding and $n_i+n_{i+1}$ vertices following it along $c$ (and possibly to more vertices). Hence, the vertex between $v_{x_i}$ and $v_{x_{i+1}}$ with the highest degree is adjacent to all other vertices in $\{v_{m_i+1},v_{m_i+2},\ldots,v_{m_{i+1}+n_{i+1}}\}$, and in particular to the vertex between $v_{x_i}$ and $v_{x_{i+1}}$ with the lowest degree. The proof is concluded by observing that the edges of the spine $(u_1,u_2,\ldots,u_s)$ do not cross each other, since the vertices $u_1, u_2,\ldots, u_s$ appear in this order along $c$.
\end{proof}

\section{Conclusions and Open Problems}

In this paper we introduced and studied the problem of constructing geometric graphs with few vertices and edges that are universal for families of planar graphs. Our research raises several challenging problems.

\paragraph{Universal geometric graphs.}
First, what is the minimum number of edges in an $n$-vertex convex geometric graph that is universal for $n$-vertex trees?
We proved that the answer is in $O(n \log n)$ if the convexity requirement is dropped, or if caterpillars, rather than trees, are considered, while the answer is close to $\Omega(n^2)$ if outerplanar graphs, rather than trees, are considered.

Second, what is the minimum number of edges in a geometric graph that is universal for all $n$-vertex planar graphs?
For abstract graphs, Babai~et~al.~\cite{bcegs-gcasg-82} constructed a universal graph with $O(n^{3/2})$ edges based on separators. Can such a construction be adapted to a geometric setting? The current best lower bound is $\Omega(n\log n)$, same as for trees~\cite{cg-ugst-83}, while the best upper bound is only $O(n^4)$.

Third, for a constant $\Delta\in\N$, what is the minimum number of edges in a geometric graph that is universal for all $n$-vertex planar graphs of maximum degree $\Delta$?

\paragraph{Plane graphs.}
The notion of universality can be further extended to \emph{plane graphs}, that is, planar graphs with given rotation. The  \emph{rotation} of a graph embedded in the plane (or in an orientable surface) is the counterclockwise orders of incident edges at all vertices. For a class $\mathcal H$ of plane graphs, a geometric graph is \emph{universal for $\mathcal H$} if it contains an embedding of every graph in $\mathcal{H}$ with the given rotation. Our upper bounds do not extend to this setting. In particular, we do not know what is the minimum number of edges (i) in a geometric graph that is universal for all $n$-vertex \emph{plane} trees, and (ii) in a convex geometric graph that is universal for all $n$-vertex \emph{plane} caterpillars.

\paragraph{Topological (multi-)graphs.}
Finally, the problems considered in this paper can be posed for \emph{topological} (multi-)graphs, as well, in which edges are represented by Jordan arcs. Within this setting we observe a sub-quartic upper bound for the number of edges of a topological multi-graph that is universal for all $n$-vertex planar graphs.

\begin{theorem}\label{thm:multi:top}
For every $n\in \N$, there exists a topological multigraph with $n$ vertices and $O(n^3)$ edges that contains a planar drawing of every $n$-vertex planar graph.
\end{theorem}
\begin{proof}
	Every planar graph admits a $2$-page monotone topological book
	embedding~\cite{ddlw-ccdpg-05}. In such a drawing, every edge is either drawn
	on one page only, or it is drawn so that it crosses the spine exactly once. As
	there are $\binom{n}{2}$ possible pairs of endpoints, two choices for the page
	incident to the left endpoint, and $n$ possible segments of the spine to
	cross, we have $n^2(n-1)=\Theta(n^3)$ edges to draw. A drawing that
	encompasses all those edges is universal for planar graphs on $n$ vertices.
\end{proof}

A planar graph is \emph{subhamiltonian} if it is a subgraph of a Hamiltonian planar graph. In particular, planar graphs of degree at most four and planar graphs that do not contain a separating triangle are subhamiltonian.

\begin{theorem}\label{thm:multi:top:H}
For every $n\in \N$, there exists a topological multigraph with $n$ vertices and
	$O(n^2)$ edges that contains a planar drawing of every
	$n$-vertex subhamiltonian planar graph.
\end{theorem}
\begin{proof}
	Every subhamiltonian planar graph admits a $2$-page book
	embedding~\cite{bk-btg-79}. On each page, draw the complete graph using
	$\binom{n}{2}$ circular arcs, and let $D$ denote the union of these two
	drawings. Clearly, $G$ contains $n(n-1)=\Theta(n^2)$ edges and is universal
	for subhamiltonian planar graphs on $n$ vertices.
\end{proof}

\bibliographystyle{plain}
\bibliography{universal}

\end{document}